\documentclass[12pt]{amsart}

\usepackage{amsfonts,amssymb}

\usepackage{amsthm}

\usepackage{verbatim}

\usepackage{color}

\usepackage{a4wide}

\theoremstyle{plain}

\newtheorem{theorem}{Theorem}[section]
\newtheorem{lemma}[theorem]{Lemma}
\newtheorem{proposition}[theorem]{Proposition}
\newtheorem{corollary}[theorem]{Corollary}
\newtheorem{Counter-example}[theorem]{Counter-example}
\newtheorem{remark}[theorem]{Remark}

\theoremstyle{definition}

\theoremstyle{remark}

\long\def\symbolfootnote[#1]#2{\begingroup\def\thefootnote{\fnsymbol{footnote}}
\footnote[#1]{#2}\endgroup}

\usepackage{amssymb}

\makeatletter
\@namedef{subjclassname@2020}{\textup{2020} Mathematics Subject Classification}
\makeatother

\begin{document}

\def\Q{\mathbb Q}
\def\R{\mathbb R}
\def\N{\mathbb N}
\def\Z{\mathbb Z}
\def\C{\mathbb C}
\def\S{\mathbb S}
\def\L{\mathbb L}
\def\H{\mathbb H}
\def\K{\mathbb K}
\def\X{\mathbb X}
\def\Y{\mathbb Y}
\def\Z{\mathbb Z}
\def\E{\mathbb E}
\def\J{\mathbb J}
\def\I{\mathbb I}
\def\T{\mathbb T}
\def\H{\mathbb H}

\title{Exploring new extrinsic upper bounds on the first eigenvalue of the Laplace operator for compact submanifolds in Euclidean spaces}

\author{Francisco J. Palomo$^*$}
\address{
Departamento de Matem\'{a}tica Aplicada, 
Universidad de M\'{a}laga,  29071-M\'{a}laga (Spain)}
\email{fpalomo@uma.es}

\author{Alfonso Romero}
\address{
Departamento de Geometr\'{i}a y Topolog\'{i}a, Universidad de Granada, 18071-Granada (Spain)}
\email{aromero@ugr.es}



\date{}

\symbolfootnote[ 0 ]{ 
$^*$ Corresponding author.
}

\begin{abstract} Upper bounds of the first non-trivial eigenvalue $\lambda_1$ of the Laplace operator of a compact submanifold $M^n$ of Euclidean space $\R^{m+1}$, by means of a new technique,  are obtained. Each of the upper bounds of $\lambda_1$ depends on the length  of mean curvature vector field, the dimension $n$, the volume of $M^n$, and of a vector  of $\R^{m+1}$. 
When $M^n$ does not lie minimally in a hypersphere of $\R^{m+1}$, classical Reilly's inequality  \cite{Re} is improved and new upper bounds are explicitly computed.  For instance, considering a torus of revolution whose generating circle has a radius of $1$ and is centered at distance $\sqrt{2}$ from the axis of revolution, we find  $\lambda_1 < \frac{4}{3}(\sqrt{2}-1)\approx 0.552284$, whereas Reilly's upper bound gives $\lambda_1 < 1/\sqrt{2}\approx 0.707106$.

\end{abstract}
\textit{\textbf{}}
\maketitle

\markboth{}{}

\thispagestyle{empty}

\hyphenation{Lo-rent-zi-an}

\section{Introduction}

This paper is motivated by the following well-known Reilly's result \cite[Main Theorem]{Re},

\vspace{1mm}

\begin{quote}{\it
{\bf Theorem (Reilly).} Let $\psi:M^{n}\rightarrow \R^{m+1}$, $n\leq m$, be a (connected) compact $n$-dimensional submanifold $M^n$ in Euclidean space $\R^{m+1}$.
The first (non-trivial) eigenvalue $\lambda_1$ of the Laplace operator of the induced metric on $M^n$ satisfies  
$$
\hspace*{-5cm}\mathrm{(Re)} \hspace{36mm}\lambda_1 \leq n\,\frac{\displaystyle{\int}_{M^n} \|\mathbf{H}\|^2 \,dV}{\mathrm{vol}(M^{n})}=: B(\psi)\,,
$$
where $\mathbf{H}$ is the mean curvature vector field of $\psi$,  ${\mathrm{vol}(M^{n})}$ is the volume of $M^n$ and $dV$ denotes the canonical measure of the induced metric. The equality holds if and only if $\psi(M^n)$ lies minimally in some hypersphere in $\R^{m+1}$.
 }
\end{quote}

\noindent Although, the Laplace operator and its first non-trivial eigenvalue $\lambda_{1}$ for a Riemannian manifold $M^n$ are intrinsic objects (see \cite{BGM}, for instance), upper bounds for $\lambda_{1}$ can be computed from extrinsic quantities relative to some isometric immersion in Euclidean space as Reilly's result shows. 
His estimate improved the earlier one  \cite[Thm. 1]{BW} by Bleecker and Weiner,  
$$
\lambda_1 \leq n\,\frac{\displaystyle{\int}_{M^n} \| \mathrm{II} \|^2 \,dV}{\mathrm{vol}(M^{n})},
$$
where $\mathrm{II}$ denotes the second fundamental form. The equality  characterizes the spheres of constant sectional curvature.

\vspace{1mm}

Observe that the upper bound for  $\lambda_1$ in Reilly's  Theorem works for all congruent submanifolds to $\psi:M^{n}\rightarrow \R^{m+1}$ and the result 
is optimal in the sense that the equality gives a characterization of a relevant property of the submanifold. When the submanifold does not lie minimally in a hypersphere in $\R^{m+1}$, inequality (Re) provides a strict upper bound for $\lambda_1$.
Reilly's inequality was extended  to other ambient
spaces in \cite{He}.

\vspace{1mm}

In this setting,  we ask ourselves about the possibility of achieving a better upper bound for $\lambda_1$ than (Re) for a submanifold that does not lie minimally in a hypersphere.
The main goals in this paper are, (Sections 4 and 5) 
\vspace{1mm}
\begin{quote}{\it
{\bf Theorem A.} Let $\psi:M^{n}\rightarrow \R^{m+1}$ be a compact $n$-dimensional submanifold $M^n$ in Euclidean space $\R^{m+1}$.
For each $v\in \R^{m+1}$, the first eigenvalue $\lambda_1$ of the Laplace operator of  $M^n$ satisfies
$$
\hspace*{-0.5cm}\mathrm{(PR1)} \hspace*{3cm}\lambda_1 \leq n\,\frac{\displaystyle{\int}_{M^n}\big(\|\mathbf{H}\|^2+\langle \mathbf{H}, v \rangle^{2} \big) \,dV}{\mathrm{vol}(M^{n}) +\dfrac{1}{n} \displaystyle{\int}_{M^n}\| v^{\top}\|^2\, dV}=:b(\psi,v)\, ,
$$
where  $v^{\top}=\nabla \langle \psi,v \rangle\in \mathfrak{X}(M^n)$ is, at any point $p\in M^n$, the orthogonal projection of $v$ on $T_{p}(M^n)$. 
Moreover, the following assertions  are equivalent:

\smallskip

\begin{enumerate}
\item The equality in $\mathrm{(PR1)}$ holds for some $v\in \R^{m+1}$.
\item The immersion $\psi$ is minimal  in a hypersphere of radius $\sqrt{n/\lambda_{1}}$ in $\R^{m+1}$. 
\item The equality in $\mathrm{(PR1)}$ holds for any $w\in \R^{m+1}$.

\end{enumerate}
}
\end{quote}

\bigskip

\begin{quote}{\it
{\bf Theorem B.}
Let $\psi:M^{n}\rightarrow \R^{m+1}$ be a  compact $n$-dimensional submanifold $M^n$ in Euclidean space $\R^{m+1}$.
For each $v\in \S^{m}$, the first (non-trivial) eigenvalue $\lambda_1$ of the Laplace operator of  $M^n$ satisfies
$$
\hspace*{-0.5cm}\mathrm{(PR2)} \hspace*{3cm}\lambda_{1}\leq n \frac{\displaystyle \int_{M^n}\big(\|\mathbf{H}\|^2- \langle \mathbf{H}, v \rangle^{2}\big)\, dV}{\mathrm{vol}(M^n)- \dfrac{1}{n}\displaystyle \int_{M^n}\| v^{\top}\|^2 \, dV}\,=: \tilde{b}(\psi,v).
$$
 If $\psi$ has its  center of gravity  located at the origin, then the equality holds for some $v\in \S^{m}$ if and only if there is 
 $\rho_v \in \mathcal{C}^{\infty}(M^n)$ with $\int_{M^n}\rho_v\, dV=0$ such that $$\Delta  \psi + \lambda_{1} \psi= \rho_v\, v.$$
}
\end{quote}

\noindent 
Note that $n\,\mathrm{vol}(M^n)>\int_{M^n} \| v^{\top} \|^{2}\, dV$ holds for any $v\in \mathbb{S}^m$ (see (\ref{23julyf})). On the other hand, for the equality condition without the assumption on the center of gravity see Remark \ref{0910A}.

\smallskip

Obviously,  Theorem A  becomes Reilly's Theorem for $v=0$. Furthermore, Theorems A and B
permit reproving inequality (Re) as an average of (PR1) and (PR2) on the unit sphere $\S^{m}$,  
see Remarks \ref{tras_Theorem_A} and \ref{tras_Theorem_5.2}.

\smallskip

The following question arises in a natural way

\vspace{1mm}

\begin{quote}{\it Given a compact submanifold $\psi : M^{n}\rightarrow \R^{m+1}$ where $\psi(M^n)$ is not minimally contained in a hypersphere of $\R^{m+1}$, is there some $v\in \R^{m+1}\setminus \{0\}$ so that
$b(\psi,v) < B(\psi)$ or $\tilde{b}(\psi,v) < B(\psi)$\,?
}
\end{quote}

\vspace{1mm}

\noindent In order to face this question, we introduce the quadratic form $Q$ on $\R^{m+1}$ given by
\begin{equation}\label{quadratic}
Q(v)=n\,\mathrm{vol}(M^n)\int_{M^n}\langle \mathbf{H}, v \rangle^{2}\, dV - \displaystyle{\int}_{M^n}\|\mathbf{H}\|^2\, dV \displaystyle{\int}_{M^n}\|v^{\top}\|^{2}\, dV.
\end{equation}
From Lemma \ref{250923A}, we have that $b(\psi,v) \leq B(\psi)$ for some vector $v\in \mathbb{R}^{m+1}$
if and only if $Q(v)\leq 0$, and $\tilde{b}(\psi,v) \leq B(\psi)$ for some $v\in \mathbb{S}^m$ 
if and only if $Q(v)\geq 0$. Each one of these inequalities holds strictly if and only if $Q(v)<0$ or $Q(v)>0$, respectively. Note that Lemma \ref{240923A} states that $Q$ is never definite and therefore, as a consequence of Theorems A and B, we obtain

\vspace{1mm}

\begin{quote}{\it {\bf Theorem C.} For any compact submanifold $\psi : M^{n}\rightarrow \R^{m+1}$, there always exists $v\in \R^{m+1}\setminus \{0\}$ $($resp. $v\in \S^{m})$ such that
$b(\psi,v) \leq B(\psi)$ $($resp. $\tilde{b}(\psi,v) \leq B(\psi))$.

}
\end{quote}

\noindent Even more, as shown in Lemma \ref{250923A}, if $Q\neq 0$ then $b(\psi, v)< B(\psi)$ (resp. $\tilde{b}(\psi, v)< B(\psi)$) whenever $Q(v)<0$ (resp. $Q(v)>0$).

\vspace{1mm}

As an application, the last section of this paper is devoted to obtain explicit upper bounds of $\lambda_{1}(\T_{R})$, where $\T_{R}$ is the following torus embedded in the  three dimensional euclidean space
$$\mathbb{T}_{R}=\{(x,y,z)\in\mathbb{R}^3\, :\, (\sqrt{x^2+y^2}-R)^2+z^2=1 \},$$
with $R>1$.  Although the set of eigenvalues of the Laplace operator for flat tori is well-known (see for instance \cite{Cha}), this is not the case for the embedded ones.  For instance and from a different point of view, the  problem of the estimation of the eigenvalues of the Laplace operator  for embedded tori has been recently faced  in \cite{Volk}. 
In order to get our upper bounds, we first obtain an explicit formula for the quadratic form given in (\ref{quadratic}), Proposition \ref{031123A}. 
This result identifies three qualitatively different behaviors for the quadratic form  corresponding to the conditions $R^{2} > 9/8$, $R^2 < 9/8$, and $R^2 = 9/8$. In the first and the second cases, we select
the specific tori $R^2=2$ and $R^{2}=17/16$ and show that the Reilly's upper bound can be improved in both tori.  In the particular,  $\lambda_1(\T_{\sqrt{2}})< \frac{4}{3}(\sqrt{2}-1)\approx 0.552284$, whereas we know $\lambda_1(\T_{\sqrt{2}})< 1/\sqrt{2}\approx 0.707106$ from (Re). Note that the Willmore functional for tori in $\R^3$ \cite{Will} attains its minimum value at  $\T_{\sqrt{2}}$.

\section{Preliminaries}
\noindent Let $\R^{m+1}$ be the $m$-dimensional Euclidean space and 
$(x_{1},\dots, x_{m+1})$ its canonical coordinates. 
For a smooth immersion $\psi:M^{n}\rightarrow \R^{m+1}$ of an
$(n\geq 2)$-dimensional (connected) manifold $M^n$ the induced metric via $\psi$ from the canonical metric $\langle\;,\;\rangle =dx_1^2+\dots dx_{m+1}^2$ of $\R ^{m+1}$ is also denoted by $\langle\;,\;\rangle$.
 
\smallskip

Let $\overline{\nabla}$ and $\nabla$ be the Levi-Civita
connections of  $\R^{m+1}$ and $M^n$, respectively, and let
$\nabla^{\perp}$ be the normal connection. The Gauss
and Weingarten formulas are written as follows
$$\overline{\nabla}_X Y=\nabla_XY + \mathrm{II}(X,Y)
\, \quad \mathrm{and} \, \quad
\overline{\nabla}_X\xi=-A_{\xi}X+\nabla^{\perp}_X\,\xi,$$
for any $X,Y \in \mathfrak{X}(M^{n})$ and $\xi \in
\mathfrak{X}^{\perp}(M^{n})$, where $\mathrm{II}$ denotes the
the second fundamental form of $\psi$. The shape operator $A_{\xi}$,
corresponding to $\xi$, is related to $\mathrm{II}$ by
$$\langle A_{\xi}X, Y \rangle = \langle \mathrm{II}(X,Y), \xi
\rangle,$$
for all $X,Y \in \mathfrak{X}(M^{n})$. 
The mean curvature vector field of $\psi$ is defined by
$\mathbf{H}=\frac{1}{n}\,\mathrm{trace}_{_{\langle\; , \;
\rangle}}\mathrm{II},$ and it satisfies the Beltrami equation 
\begin{equation}\label{beltrami}
\Delta \psi=n \mathbf{H},
\end{equation} where $\Delta$ is the Laplace operator of $M^n$, $\Delta \psi =(\Delta \psi_{1} ,\cdots , \Delta \psi_{m} )$ and $\psi = (\psi_{1} ,\cdots ,\psi_{m+1})$. Moreover, we have $\Delta\| \psi\|^{2}= 2n+ 2n\langle \mathbf{H},\psi \rangle$ and therefore, being $M^n$ compact, we have the well-known Minkowski formula
\begin{equation}\label{1040}
\int_{M^n}(1+\langle \mathbf{H}, \psi \rangle)\,dV=0.
\end{equation}

\section{The key tool}
\noindent Let $\L^{m+2}$ be the $(m+2)$-dimensional Lorentz-Minkowski spacetime; i.e.,  $\L^{m+2}$ is $\R^{m+2}$ endowed with the Lorentzian metric
$
g =-dx_{0}^{2}+ dx_{1}^{2}+ \cdots + dx_{m+1}^2.
$
An immersion $\phi:M^{n}\rightarrow \L^{m+2}$ of an
$n$-dimensional manifold $M^n$ in $\L^{m+2}$ is said to be spacelike when the induced metric on $M^n$ via $\psi$ is Riemannian. We also denote by $g$ the induced metric on $M^n$.  If we consider the totally geodesic embedding $i\colon \R^{m+1}\to \L^{m+2}$ given by $i(x)=(0,x)$, each  submanifold of the Euclidean space $\R^{m+1}$ can be naturally seen as a spacelike submanifold of $\L^{m+2}$.

\smallskip

The following result will be the key tool for this paper \cite[Theor. 6.10]{PaRo1}
\begin{theorem}\label{170123A}
For each unit timelike vector $a\in \L^{m+2}$ $($i.e., such that $g(a,a)=-1)$, the first (non-trivial) eigenvalue $\lambda_1$ of the Laplace operator of a compact $n$-dimensional spacelike submanifold $M^n$ in $\L^{m+2}$, $m\geq n$, satisfies
$$
\hspace*{-45mm}\mathrm{(L)} \hspace*{4.5cm}\lambda_1 \leq n\,\frac{\displaystyle{\int}_{M^n}\big(\|\mathbf{H}\|_{g}^2 + g (\mathbf{H},a)\big)\,dV}{\mathrm{vol}(M^n) +\dfrac{1}{n} \displaystyle{\int}_{M^n}\| a^{\top}\|_{g}^2\, dV}\,,
$$
where  $a^{\top}\in \mathfrak{X}(M^n)$ is, at any point $p\in M$, the orthogonal projection of the vector $a$ on $T_{p}(M^n)$. If $\phi:M^n \to \L^m$ has the center of gravity located at the origin, the equality in $\mathrm{(L)}$ holds if and only if there exists $\mu_{a}\in C^{\infty}(M^n)$ such that
$
\Delta \phi + \lambda_{1} \phi=\mu_{a}\, a.
$
\end{theorem}

\begin{remark}
{\rm If $\phi$ has its gravity center located at the point $c=(c_{0}, \cdots, c_{m+1})\in \L^{m+2}$ where $c_{i}=\int_{M^n}\phi_{i}\, dV$, then the spacelike immersion
\begin{equation}\label{translation}
\widehat{\phi}:=\phi-\frac{1}{\mathrm{vol}(M^n)}\, c
\end{equation}
has the center of gravity located at the origin. Hence, the equality holds in (L)
if and only if
\begin{equation}\label{equality in (PR)}
\Delta \phi + \lambda_{1} \phi-\mu_{a} a=\frac{\lambda_{1}}{\mathrm{vol}(M^n)}\, c \in \L^{m+2}.
\end{equation}
Note that by integrating this formula, we get 
\begin{equation}\label{mean_zero}
\int_{M^n}\mu_{a}\, dV=0.
\end{equation}

Conversely, assume now $\Delta \phi + \lambda_{1} \phi-\mu_{a} a=b\in \L^{m+2}$  with $
\int_{M^n}\mu_{a}\, dV=0\,.
$
Then, we have that $b=(\lambda_{1}/{\mathrm{vol}(M^n)})\, c .$ 

\smallskip

In other words, the equality in (L) holds if and only if there is a function $\mu_{a}\in C^{\infty}(M^n)$  with $
\int_{M}\mu_{a}\, dV=0$ and $b\in \L^{m+2}$ such that $\Delta \phi + \lambda_{1} \phi-\mu_{a} a=b\in \L^{m+2}$. 
}
\end{remark}

\section{Proof of Theorem A}
\noindent 
For a compact submanifold $\psi : M^n \rightarrow \R^{m+1}$, let us consider the spacelike submanifold $\phi : =i\circ \psi : M^n \rightarrow \L^{m+2}$. Then, inequality (PR1) for $v\in \R^{m+1}$ is a direct consequence of inequality (L) for the unit timelike vector $a=(\sqrt{1+\|v\|^{2}}, v)\in \L^{m+2}$.
Now, we deal with the equivalent conditions for equality in (PR1).

\smallskip

$(1)\Rightarrow (2)$.
Assume the equality holds in (PR1) for some $v\in \R^{m+1}$. Then, we achieve equality in (L) for  $a=(\sqrt{1+\|v\|^{2}}, v)$. Hence, there exist $\mu_{a}\in C^{\infty}(M^n)$ with $\int_{M^n}\mu_{a}\, dV=0$ and  $b\in \L^{m+2}$ such that  $\Delta \phi + \lambda_{1} \phi-\mu_{a} a=b$.
Taking into account that $\phi_{0}=0$ and $a_{0}\neq 0$, we get $b_{0}=0$ and $\mu_{a}=0$, and therefore
\begin{equation}\label{170123B}
\Delta \psi + \lambda_{1} \psi= (b_{1}, \cdots, b_{m+1})\in \R^{m+1}.
\end{equation}
For the immersion $\widehat{\psi}:= \psi- \frac{1}{\lambda_{1}}\,(b_{1}, \cdots, b_{m+1})$,  previous formula  implies that 
$\Delta \widehat{\psi} + \lambda_{1} \widehat{\psi}=0$. The Takahashi classical result in \cite{Taka} gives that $\widehat{\psi}$ realizes a minimal immersion in a hypersphere of radius $\sqrt{n/\lambda_{1}}$ with center located at the origin in $\R^{m+1}$.  Now, the assertion on $\psi$ is clear.

\vspace{1mm}
$(2)\Rightarrow (3)$.
Assume $\psi$ realizes a minimal immersion in a hypersphere. Without loss of generality, we can assume the center of gravity of $\psi$ is located at the origin. Again from the above-mentioned Takahashi result, we have  $\Delta \psi + \lambda_{1} \psi=0$.  Hence, by means of Beltrami equation (\ref{beltrami}) and for for every $w\in \R^{m+1}$, we get 
$$
\triangle \langle \psi,  w\rangle^2= 2\langle \psi, w\rangle \triangle \langle \psi, w\rangle+ 2\|\nabla \langle \psi, w\rangle\|^{2}=-\frac{1}{\lambda_{1}}2n^2 \,\langle \mathbf{H}  , w\rangle^{2}+ 2\,\| w^{\top}\|^{2}\,.
$$
Consequently, the divergence theorem gives
\begin{equation}\label{int2}
\int_{M^n}\| w^{\top}\|^{2}\, dV=\frac{n^2}{\lambda_{1}}\int_{M^n}\langle \mathbf{H} , w\rangle^{2}\, dV.
\end{equation}
On the other hand, we call again Takahashi result \cite{Taka} and Beltrami equation (\ref{beltrami})  to show $\| \mathbf{H}\|^2 =\lambda_{1}/n$.
Therefore, from these facts and (\ref{int2}) we get 
$$
b(\psi,w)=n\,\frac{\dfrac{\lambda_1}{n}\,\mathrm{vol}(M^n)+\int_{M^n}\langle \mathbf{H}, w \rangle^{2}  \,dV}{\mathrm{vol}(M^{n}) +\dfrac{n}{\lambda_{1}}\int_{M^n}\langle \mathbf{H} , w\rangle^{2}\, dV}\,=\lambda_{1},
$$
for every $w\in \R^{m+1}$.

\vspace{1mm}
$(3)\Rightarrow (1)$ is trivial.  \begin{flushright} $\square$ \end{flushright} 

\smallskip

We end this section showing that inequality (Re) can be achieved as an average of (PR1) on the sphere $\S^{m}$. First, let us recall the well-known technical result
\begin{lemma}\label{integration_on_the_sphere} The restriction $f$ on $\S^m$ of a quadratic form $F$ on $\R^{m+1}$ satisfies
\begin{equation}
\displaystyle{\int}_{\S^m}f\,d\mu=\frac{1}{m+1}\,\mathrm{trace}_{\langle\; ,\;\rangle}(F)\,\mathrm{vol}(\S^{m})
\end{equation}
where $\mathrm{trace}_{\langle\; ,\;\rangle}(F)$ is the trace of the self-adjoint operator of $\R^{m+1}$ defined by $F$ using the usual metric $\langle\; ,\,\rangle$ of $\R^{m+1}$.
\end{lemma}

\begin{remark}\label{tras_Theorem_A}{\rm 
Integrating the inequality in Theorem A on the sphere $\S^{m}$ by means of the Fubini's theorem, we 
get
\begin{equation}\label{270223Fa}
\lambda_{1}\left[\mathrm{vol}(M^n)\mathrm{vol}(\S^{m})+ \frac{1}{n}\int_{M^n}\Big(\int_{\S^{m}}\| v^{T}\|^2\, d\mu \Big) \, dV\, \right]\leq 
\end{equation}
$$
n\, \mathrm{vol}(\S^{m})\int_{M^n}\|\mathbf{H}\|^2\, dV+ \int_{M^n}\Big(\int_{\S^{m}}\langle \mathbf{H}, v \rangle^{2}\, d\mu\Big)\, dV\,.
$$
On the other hand, we call Lemma \ref{integration_on_the_sphere} to get
\begin{equation}\label{integration_on_the_sphere2a}
\int_{\S^{m}}\langle \mathbf{H}, v \rangle^{2}\, d\mu=\frac{\mathrm{vol}(\S^{m})}{m+1}\| \mathbf{H}\|^2\quad  \textrm{ and }\quad 
\int_{\S^{m}}\| v^{\top}\|^2\, d\mu=\frac{n\,\mathrm{vol}(\S^{m})}{m+1}\,,
\end{equation}
where for the second integral we have also used
$
\triangle \langle \psi, e_{i}\rangle^{2}= 
2n\langle \psi, e_{i}\rangle  \langle \mathbf{H}, e_{i}\rangle+2\| \nabla \langle \psi, e_{i}\rangle\|^{2}$, $1\leq i \leq m+1$, as well the Minkowski integral formula (\ref{1040}).

\vspace{1mm}

Taking into account (\ref{integration_on_the_sphere2a}), inequality (\ref{270223Fa}) becomes inequality (Re).  

}
\end{remark}

\smallskip

\section{Proof of Theorem B. A variation on Reilly's argument}
\noindent In order to proof the inequality (PR2), no generality is lost if we assume the center of gravity of the compact submanifold $\psi:M^{n}\rightarrow \R^{m+1}$  is located at the origin.
Thus, every component of $\psi$ satisfies
\begin{equation}\label{media_nula}
\int_{M^n}\psi_{j}\, dV=0, \quad j=1, \cdots, m+1.
\end{equation}
According to the  Minimum Principle for the smallest positive eigenvalue $\lambda_{1}$ of the Laplace operator $\Delta$ of $M^n$ \cite[p. 186]{BGM} we know,
\begin{equation}\label{minimum}
 \lambda_{1}\int_{M^n}f^2\, dV \leq \int_{M^n}\| \nabla f \|^2 dV\,,
\end{equation}
for any non-zero $f\in C^1(M^n)$ with $\int_{M^n}fdV=0$,
and the equality holds if and only if $f$ is an eigenfunction of $\Delta$ corresponding to $\lambda_1$.
Thus, we have
\begin{equation}\label{v_sphere2}
\lambda_{1}\int_{x\in M^n}\langle \psi(x), v \rangle^{2} \, dV \leq \int_{x\in M^n} \|\nabla \langle \psi, v\rangle (x) \|^{2}\, dV\,,
\end{equation}
for every $v\in \R^{m+1}$. 

\smallskip

Now, for each $v\in \S^{m}$, we consider
\begin{equation}\label{v_sphere}
\S^{m-1}_{v}:=\{w\in \R^{m+1}: \|w\|^{2}=1\textrm{ and } \langle v,w \rangle=0\}\,,
\end{equation}
which is isometric to the $(m-1)$-dimensional unit sphere in $\R^m$.
Making use of Fubini's theorem, we can integrate on $\S^{m-1}_{v}$ both members of inequality (\ref{v_sphere2})  to get
\begin{equation}\label{260223A}
\lambda_{1}\int_{x\in M^n}\Big[\int_{w\in \S^{m-1}_{v}}\langle \psi(x), w \rangle^{2} \, d\mu_v\Big]\, dV \leq \int_{x\in M^n}\Big[\int_{w\in \S^{m-1}_{v}} \|\nabla \langle \psi, w\rangle (x)  \|^{2}\, d\mu_v\Big]\, d V,
\end{equation}
where $d\mu_v$ denotes the canonical metric measure on $\S^{m-1}_{v}$.

\vspace{1mm}

To obtain  more explicitly the previous inequality,
we need the following averaging principle for the hypersphere $\S^{m-1}_v$ in the hyperplane $v^{\perp}=: \R^{m}_{v}\subset \R^{m+1}$.
For a quadratic form $F$ defined on $\R^{m+1}$, we have
\begin{equation}\label{25julya}
\int_{\S^{m-1}_{v}}\,F \, d\mu_v=\frac{1}{m}\,\big(\mathrm{trace}_{\langle\, ,\, \rangle}(F)- F(v)\big)\,\mathrm{vol}(\S^{m-1}).
\end{equation}
This formula is a direct consequence of Lemma \ref{integration_on_the_sphere}.
\bigskip

We are now in a position to compute the involved integrals in (\ref{260223A}).
For every $x\in M$,
we define the following bilinear form on $\R^{m+1}$,
$
T_{x}(u,w)=\langle \psi(x), u \rangle\langle \psi(x), w \rangle\,
$
and consider the corresponding quadratic form $F_{x}$. Then the  averaging principle (\ref{25julya}) gives
\begin{equation}\label{23july23}
\int_{w\in \S^{m-1}_{v}}\langle \psi (x), w \rangle^{2} \, d\mu_v=\frac{1}{m}\,\big(\| \psi(x) \|^{2}- \langle \psi(x), v\rangle^{2}\big)\,\mathrm{vol}(\S^{m-1}).
\end{equation}
In a similar way, we get
\begin{equation}\label{23july23a}
\int_{w\in \S^{m-1}_{v}} \|\nabla \langle \psi, w\rangle (x)  \|^{2}\, d\mu_v=\frac{1}{m}\,\big(n- \|v^{\top}(x)\|^{2}\big)\,\mathrm{vol}(\S^{m-1}).
\end{equation}
Therefore, using (\ref{23july23}) and (\ref{23july23a}), the inequality (\ref{260223A}) can be written as follows,
\begin{equation}\label{270223D}
\lambda_{1}\int_{M^n}\big[\| \psi \|^{2}- \langle \psi, v\rangle^{2}\big]\, dV \leq n\,\mathrm{vol}(M^n)-\int_{ M^n}\|v^{\top}\|^2\, d V,
\end{equation}
with equality if and only if $\triangle \langle \psi, w \rangle+ \lambda_{1}\langle \psi, w \rangle =0$ for every $w\in \S^{m-1}_{v},$ i.e.,
$
\triangle  \psi + \lambda_{1} \psi= \rho_{v} \, v
$
for some $\rho_{v} \in \mathcal{C}^{\infty}(M^n)$ with $\int_{M^n}\rho_{v}\, dV=0.$

\vspace{2mm}

Next, we consider the orthogonal projection
$$
\mathcal{P}_{v}\colon \R^{m+1}\to \R^{m}_{v}, \quad  w \mapsto w- \langle w, v \rangle v.
$$
Thus, we have  $\|\mathcal{P}_{v}(\psi)\|^2=\| \psi \|^{2}- \langle \psi, v\rangle^{2}$
and $\|\mathcal{P}_{v}(\mathbf{H})\|^2=\| \mathbf{H} \|^{2}- \langle \mathbf{H}, v\rangle^{2}$.

\vspace{2mm}

Using the $L^2(M^n)$ Schwarz inequality as well the Schwarz inequality in $\R^{m}_{v}$ we obtain,
 \begin{equation}\label{270223C}
\int_{M^n}\|\mathcal{P}_{v}(\psi)\|^2\, dV\, \int_{M^n}\|\mathcal{P}_{v}(\mathbf{H})\|^2\, dV\geq \Big( \int_{M^n}\langle \mathcal{P}_{v}(\psi), \mathcal{P}_{v}(\mathbf{H})\rangle\, dV\Big)^2.
\end{equation}
Observe that, making use of Minkowski formula (\ref{1040}), we can write,
\begin{equation}\label{23julyb}
\int_{M^n}\langle \mathcal{P}_{v}(\psi), \mathcal{P}_{v}(\mathbf{H})\rangle\, dV = \int_{M^n}(\langle \psi, \mathbf{H}\rangle- \langle \psi, v\rangle\langle \mathbf{H}, v \rangle)\, dV
\end{equation}
$$
\hspace*{48mm}=- \mathrm{vol}(M^n)-\int_{M^n} \langle \psi, v\rangle\langle \mathbf{H}, v \rangle\, dV\,.
$$
Furthermore, from (\ref{beltrami}), we have one more time
$
\triangle \langle \psi, v \rangle^{2}= 2n\langle \psi, v \rangle \langle \mathbf{H}, v \rangle+ 2 \|\nabla \langle \psi, v \rangle\|^2\,,
$
and then
\begin{equation}\label{23julyc}
\int_{M^n} \langle \psi, v\rangle\langle \mathbf{H}, v \rangle\, dV=-\frac{1}{n}\int_{M^n}\|v^{\top}\|^2\, dV.
\end{equation}
Hence, if we use (\ref{23julyb}) and (\ref{23julyc}), then (\ref{270223C}) implies that
\begin{equation}\label{23julyd}
  \Big(\mathrm{vol}(M^n) -\frac{1}{n}\int_{M^n}\|v^{\top}\|^2\, dV\Big)^2\leq \int_{M^n}\|\mathcal{P}_{v}(\psi)\|^2\, dV\, \int_{M^n}\|\mathcal{P}_{v}(\mathbf{H})\|^2\, dV.
\end{equation}

From (\ref{23julyd}), the inequality (\ref{270223D}) gives
\begin{equation}\label{23julye}
\lambda_{1}\,\Big(\mathrm{vol}(M^n)- \frac{1}{n}\int_{M^n}\| v^{\top}\|^2 \, dV\Big)^{2}\leq \Big(n\,\mathrm{vol}(M^n)-\int_{ M^n}\|v^{\top}\|^2\, d V \Big)\int_{M^n}\|\mathcal{P}_{v}(\mathbf{H})\|^2\, dV.
\end{equation}
Lastly, we observe that $\|v^{T}\|^2<1$ and then,
\begin{equation}\label{23julyf}
\mathrm{vol}(M^n)- \frac{1}{n}\int_{M^n}\| v^{\top}\|^2 \, dV>0\,.
\end{equation}
Therefore,  the announced inequality in Theorem B is reached from (\ref{23julye}).

\vspace{2mm}

The equality holds in (PR2) if and only if we have the equality in (\ref{270223D}).
That is, $
\triangle  \psi + \lambda_{1} \psi= \rho_{v} \, v
$
for some $\rho_{v} \in \mathcal{C}^{\infty}(M^n)$ with $\int_{M^n}\rho_{v}\, dV=0.$
\begin{flushright} $\square$ \end{flushright} 

\begin{remark}
{\rm The proof of Theorem B is inspired by the original Reilly's argument in \cite[Main Theorem]{Re}.
As far as we know, the proof of Theorem A cannot be achieved from a similar variation on Reilly's argument. 

}
\end{remark}

\begin{remark}\label{0910A}
{\rm If we remove the assumption on the center of gravity of $\psi$ in Theorem B, the equality condition for (PR2) reads as 
$$
\triangle \psi+\lambda_{1}\psi=\rho_{v} v+ \frac{\lambda_{1}}{\mathrm{vol}(M^n)}c\,,
$$
where $c$ is the center of gravity of $\psi$, and as above $\rho_{v} \in \mathcal{C}^{\infty}(M^n)$ with $\int_{M^n}\rho_{v}\, dV=0.$
   }
\end{remark}

\begin{remark}\label{tras_Theorem_5.2}{\rm 
As in Remark \ref{tras_Theorem_A},  Reilly's inequality (Re) is also contained in Theorem B as an average principle. In fact, by integrating the inequality in Theorem B on the sphere $\S^{m}$ we 
get
\begin{equation}\label{270223F}
\lambda_{1}\Big(\mathrm{vol}(M^n)\mathrm{vol}(\S^{m})- \frac{1}{n}\int_{M^n}\Big(\int_{\S^{m}}\| v^{\top}\|^2\, d\mu \Big) \, dV\, \Big)\leq 
\end{equation}
$$
n\, \mathrm{vol}(\S^{m})\int_{M^n}\|\mathbf{H}\|^2\, dV- \int_{M^n}\Big(\int_{\S^{m}}\langle \mathbf{H}, v \rangle^{2}\, d\mu\Big)\, dV\,.
$$
Taking into account (\ref{integration_on_the_sphere2a}), the inequality (\ref{270223F}) becomes (Re). 

\vspace{2mm}

Note that the equality holds in (Re) if and only if the equality in Theorem B  is achieved for every $v\in \S^{m}$. That is, if and only if $\triangle \psi +\lambda_{1}\psi=0$. This is equivalent to $M^n$ lies minimally in some hypersphere in $\R^{m+1}$, using again Takahashi's result \cite{Taka}.
}
\end{remark}

\section{Proof of Theorem C} 
 
\noindent Coming back to the quoted question in the Introduction, we want to study under what assumptions $b(\psi,v)< B(\psi)$ or $\tilde{b}(\psi,v)< B(\psi)$?

\vspace{1mm}

The following result summarizes several properties of the quadratic form $Q$ in  (\ref{quadratic}). 
\begin{lemma}\label{250923A}
    
       $(1)$ We have 
        $$
        B(\psi)- b(\psi,v)= \frac{-\,Q(v)}{\mathrm{vol}(M^n)\left(\mathrm{vol}(M^n)+ \dfrac{1}{n}\displaystyle\int_{M^{n}}\|v^{\top}\|^{2}\, dV\right)}\,.
        $$
       Therefore, the inequality $b(\psi,v)\leq  B(\psi)$ $($resp. $b(\psi,v)< B(\psi))$ holds for some $v\in \R^{m+1}$ if and only if $Q(v)\leq 0$ $($resp. $Q(v)<0)$. 
        Moreover, if $v\in \S^m$ then $b(\psi,v)\leq \tilde{b}(\psi,v)$ holds true.

        \vspace{1mm}
         $(2)$ We have 
         $$
        B(\psi)- \tilde{b}(\psi,v)= \frac{Q(v)}{\mathrm{vol}(M^n)\left(\mathrm{vol}(M^n)- \dfrac{1}{n}\displaystyle\int_{M^{n}}\|v^{\top}\|^{2}\, dV\right)}\,.
        $$
        Therefore, the  inequality  $\tilde{b}(\psi,v)\leq B(\psi)$ $($resp. $\tilde{b}(\psi,v)< B(\psi))$ holds for some $v\in \mathbb{S}^m$ if and only  $Q(v)\geq 0$ $($resp. $Q(v)>0))$. In both cases, $\tilde{b}(\psi,v)\leq b(\psi,v)$ holds true.

\end{lemma}

\begin{lemma}\label{240923A}
For every compact $n$-dimensional submanifold  $\psi:M^{n}\rightarrow \R^{m+1}$, the quadratic form $Q$ cannot be definite.
\end{lemma}
\begin{proof} 
By contradiction, suppose for instance that $Q$ is positive definite. Then, for all $v\in \R^{m+1}\setminus \{0\}$, we have
\begin{equation}\label{190123A}
n\,\mathrm{vol}(M^n)\int_{M^n}\langle \mathbf{H}, v \rangle^{2}\, dV >\displaystyle{\int}_{M^n}\|\mathbf{H}\|^2\, dV \displaystyle{\int}_{M^n}\|v^{\top}\|^{2}\, dV.
\end{equation}
Integrating both members of (\ref{190123A}) on the unit sphere $\S^{m}\subset \R^{m+1}$, we obtain
\begin{equation}\label{contradiction1}
\int_{v\in \S^{m}}\Big[ \int_{M^n}\langle \mathbf{H}, v \rangle^{2}\, dV\Big]\,d\mu=\int_{M^n}\Big[ \int_{v\in \S^{m}}\langle \mathbf{H}, v \rangle^{2}\, dV\Big]\,d\mu= \frac{\mathrm{vol}(\S^m)}{m+1}\int_{M^n} \|\mathbf{H}\|^{2}\, dV\, ,
\end{equation}
where we have made use of the first integral formula in (\ref{integration_on_the_sphere2a}). In a similar way, we get
\begin{equation}\label{contradiction2}
\int_{v\in \S^{m}}\Big[ \int_{M^n}\| v^{\top}\|^{2}\, dV\Big]\,d\mu=\int_{M^n}\Big[ \int_{v\in \S^{m}}\| v^{\top}\|^{2}\, dV\Big]\,d\mu =\frac{n\,\mathrm{vol}(\S^m)}{m+1}\mathrm{vol}(M^n)
\end{equation}
Hence, when the inequality (\ref{190123A}) holds for any $v\in \S^{m}$, we arrive to
 a contradiction.
The same proof works if we assume that $Q$ is negative definite.
\end{proof}

\vspace{1mm}

Theorem C is a direct consequence of Lemmas \ref{250923A} and \ref{240923A}.

\begin{remark}
    {\rm The quadratic form $Q$ can be identically vanished. For instance, this is the case of the round sphere $\S^{m}\subset \R^{m+1}$. This fact is checked as a direct application of Lemma \ref{integration_on_the_sphere}. The property $Q=0$ does not characterize the round sphere among the hypersurfaces in Euclidean spaces. As we will show in the following Section, there are  hypersurfaces in Euclidean spaces with the corresponding quadratic form $Q=0$ (see the case $R^2=9/8$ in Section 7). 
    
    }
\end{remark}

\begin{corollary}\label{250923B}
Let $\psi:M^{n}\rightarrow \R^{m+1}$ be a compact $n$-dimensional submanifold and $v\in \R^{m+1}\setminus \{0\}$ such that $Q(v)<0$. Then, $v^{\top}\neq 0$ and the function $f(t):= B(\psi)+b(\psi, tv)$  for $t\in [0, +\infty)$, satisfies
$f(0)=B(\psi)
$
and  is strictly decreasing. In particular, we have that
\begin{equation}\label{200123A}
\lambda_{1}\leq \lim_{t\to +\infty}f(t)=n^2\,\frac{\displaystyle{\int}_{M^n}\langle \mathbf{H}, v \rangle^{2}  \,dV}{ \displaystyle{\int}_{M^n}\| v^{\top}\|^2\, dV}< B(\psi)\,.
\end{equation}
\end{corollary}
\begin{proof} 
A direct computation gives
$$
f'(t)= \frac{2\,Q(v)\,t}{\Big(\mathrm{vol}(M^{n}) +\dfrac{t^2}{n} \displaystyle{\int}_{M^n}\| v^{\top}\|^2\, dV\Big)^2}< 0\,,
$$
for all $t\in (0,\infty)$ which shows the result. 
\end{proof}

\section{Examples: Embedded tori in $\R^3$ }
\noindent Fix a real number $R>1$,  and consider the revolution torus $\T_{R}$ obtained when the profile curve is the circle in the plane $x=0$ of center $(0,R,0)$ with a radius of $1$. First of all, we have
\begin{proposition}\label{031123A}
  The quadratic form $Q_{R}$ corresponding to $\T_{R}$ is given by
  
  \begin{equation}\label{311023A}
      Q_{R}(v)= \pi^{4}R^{3}\Big(\dfrac{4R^2-3}{\sqrt{R^2-1}}-4R\Big)\Big(v_{1}^2+v_{2}^{2}-2v_{3}^2\Big),
  \end{equation}
  where $v=(v_1,v_2,v_3)\in\R^3$.
\end{proposition}
\begin{proof}
  A local parametrization of $\T_{R}$ is
$$
\mathbf{x}(u,v)=\Big((R+ \cos u)\cos v, (R+\cos u)\sin v,  \sin u \Big), \quad (u,v)\in U=]0, 2\pi[\times ]0, 2\pi[,
$$
which covers a dense open subset $\mathbf{x}(U) \subset \T_{R}$.
The corresponding unit normal vector field is
$$
\mathbf{N}(u,v)=\Big( \cos v \cos u, \sin v \cos u, \sin u \Big),
$$
and the coefficients of the first and second fundamental forms, $E,F,G$ and $e,f,g$ respectively, satisfy
$$
E=1, \quad F=0, \quad G=(R+ \cos u)^2
$$
$$
e=-1,\quad f=0, \quad g=-(R+ \cos u)\cos u\,.
$$
Therefore, we obtain $\mathrm{area}(\T_{R})=4\pi^2 R$, and for the mean curvature function the following expression 
$$
H=\frac{1}{2}\frac{e\,G-2f\,F+g\,E}{EG-F^2}=-\frac{1}{2}\Big(1+ \frac{\cos u}{R+ \cos u}\Big).
$$
Hence, we have
\begin{equation}\label{101023A}
\int_{\T_{R}}\| \mathbf{H}\|^{2}\, dA=\frac{\pi}{2} \Big(2\pi R + \int_{0}^{2\pi}\frac{\cos^2 u}{R+ \cos u} du \Big),
\end{equation}
where $dA$ is the canonical measure on $\T_{R}$.
From (\ref{101023A}) and the  formula 
\begin{equation}\label{101023C}
    \int_{0}^{2\pi}\frac{\cos^2 u}{R+ \cos u} du=2\pi R\Big(\frac{R}{\sqrt{R^2 - 1}}-1\Big),
\end{equation}
we get, for the mean curvature vector field $
\mathbf{H}= H \mathbf{N}
$,
\begin{equation}\label{101023B}
    \int_{\T_{R}}\| \mathbf{H}\|^{2}\, dA=\frac{\pi^2 R^2}{\sqrt{R^2-1}}.
\end{equation}

\smallskip

Now, for  every $v\in \R^3$, we compute
$$
\langle \mathbf{H}, v \rangle^{2}=\frac{1}{4}\Big(1+\frac{\cos u}{R+ \cos u}\Big)^2\Big( v_{1}\cos v \cos u +v_{2}\sin v \cos u+ v_{3} \sin u \Big)^2,
$$
and therefore, 
$$
\int_{\T_{R}}\langle \mathbf{H}, v\rangle^2\, dA=
\frac{\pi}{4}\Big[(\|v\|^{2}+v_{3}^2)R\pi+(v_{1}^2+v_{2}^{2} )\int_{0}^{2\pi}\frac{\cos^4 u}{R+ \cos u} du +2v_{3}^2\int_{0}^{2\pi}\frac{\cos^2 u \sin^2 u}{R+ \cos u} du\Big].
$$
On the other hand, taking into account
$$
\int_{0}^{2\pi}\frac{\cos^4 u }{R+ \cos u} du=\pi R \Big[ 2R^2\Big(\frac{R}{\sqrt{R^2 - 1}}-1\Big)-1\Big],
$$
and (\ref{101023C}), we have
\begin{equation}\label{101023D}
    \int_{\T_{R}}\langle \mathbf{H}, v\rangle^2\, dA=\frac{\pi^{2}R^2}{2}\Big[R\Big(\frac{R}{\sqrt{R^2 - 1}}-1\Big)(v_{1}^2+v_{2}^{2})+2\Big(R-\sqrt{R^2-1}\Big)v_{3}^{2} \Big].
\end{equation}

Lastly, we decompose
$
v=v^{\top}+ v^{\perp}= v^{\top}+ h \mathbf{N},
$
where 
$$
h= \langle \mathbf{N}, v\rangle=v_{1}\cos v \cos u + v_{2}\sin v \cos u + v_{3} \sin u,
$$
and then, we obtain 
$$
v^{\top}=v-(v_{1}\cos v \cos u + v_{2}\sin v \cos u + v_{3} \sin u) \Big( \cos v \cos u, \sin v \cos u, \sin u \Big).
$$
Therefore, we get
$$
\|v^{\top}\|^{2}=\|v\|^2- (v_{1}\cos v \cos u + v_{2}\sin v \cos u + v_{3} \sin u)^2
$$
and a direct computation shows.

\begin{equation}\label{101023E}
    \int_{\T_{R}}\|v^{\top}\|^{2}\,dA=\pi^2 R(3 \|v\|^{2}- v_{3}^2).
\end{equation}

 Once we have computed the ingredients of the quadratic form $Q_{R}$ in (\ref{101023B}), (\ref{101023D}) and (\ref{101023E}), we end the proof.

\end{proof}

Note that, from (\ref{101023B}), the Reilly upper bound for the torus $\T_R$ is 
\begin{equation}\label{161123A}
 B(\T_R)=\frac{R}{2 \sqrt{R^2 -1}}.   
\end{equation}

\smallskip

Proposition \ref{031123A} identifies three qualitatively different behaviors for the quadratic form $Q_{R}$ corresponding to the conditions $R^{2} > 9/8$, $R^2 > 9/8$, and $R^2  =9/8$. In each of these cases, we select a specific torus and  compute our upper bound for $\lambda_1$, showing 
its improvement over the corresponding Reilly bounds in the first and the second cases.

\bigskip

\smallskip

\textbf{Case $R^2= 2$.}  The quadratic form is  
$$
 Q_{\sqrt{2}}(v)=2\pi^{4}(5\sqrt{2}-8)\Big(v_{1}^{2}+ v_{2}^{2}-2v_{3}^2\Big)
 $$
 for all $v\in \R^3$.

 We make first the choice  $v\in \S^{2}$ with $v_{3}^2<1/3$. In this case,  $Q_{\sqrt{2}}(v)<0$, and therefore from
  Corollary  \ref{250923B}, we have 
\begin{equation}\label{171123A}
   \lambda_{1}(\T_{\sqrt{2}})\leq \dfrac{4(\sqrt{2}-1) + v_{3}^2(12-8 \sqrt{2})}{3- v_{3}^2}. 
\end{equation}
 Now, if $v\in \S^{2}$ satisfies $v_{3}^{2}\geq  1/3$, we have  $Q_{\sqrt{2}}(v)\geq  0$ and then, Lemma \ref{250923A} (2) gives 
\begin{equation}\label{171123B}
  \lambda_{1}(\T_{\sqrt{2}})\leq \frac{\sqrt{2}}{2} -\frac{(5\sqrt{2}-8)(1-3v_{3}^2)}{10+2v_{3}^{2}}.
 \end{equation}
Note that (\ref{171123B}) for $v_{3}^2= 1/3$ is nothing but inequality (Re), namely $\lambda_{1}(\T_{\sqrt{2}})< \sqrt{2}/2 \approx 0.707106$, and for the choice $v_{3}^2 >1/3$, we have strict inequality in (\ref{171123B}).  
Observe that the minimum value of the right side  in (\ref{171123A}) under $v_{3}^2 <1/3$ (resp. in (\ref{171123B}) under $v_{3}^2 >1/3$) is attained at each $v\in \S^2$ satisfying $v_{3}=0$ (resp. $v^2_{3}=1$) and both provide
$$
\lambda_{1}(\T_{\sqrt{2}})< \frac{4}{3}(\sqrt{2}-1)\approx 0.552284.
$$

\smallskip

\noindent  We point out that the torus $\T_{\sqrt{2}}$ satisfies $\int_{\T_{\sqrt{2}}}\|\mathbf{H}\|^{2}\, dA= 2\pi^2$
from (\ref{101023B}). Thus, this torus achieves the lower bound preset in that Willmore conjectured in 1965 \cite{Will}, positively solved  in \cite{MN}.

\smallskip

\textbf{Case $R^2= 17/16$.}
 The quadratic form is 
$$
 Q_{\sqrt{17}/4}(v)=\frac{17}{64}\pi^4(5\sqrt{17}-17)\Big(v_{1}^{2}+ v_{2}^{2}-2v_{3}^2\Big)
 $$
 for all $v\in \R^3$.

 We make first the choice  $v\in \S^2$ with $v_{3}^2> 1/3$ we have  $Q_{\sqrt{17}/4}(v)<0$ and then from
  Corollary  \ref{250923B}, we have
\begin{equation}\label{171123C}
  \lambda_{1}(\T_{\sqrt{17}/4})\leq \frac{17 (\sqrt{17}-1)+ v_{3}^{2}(51-19 \sqrt{17})}{24-8v_{3}^2}.  
\end{equation}
Now, if $v\in \S^2$ satisfies $v_{3}^{2}\leq 1/3$, we have  $Q_{\sqrt{17}/4}(v)\geq  0$ and then, Lemma \ref{250923A} (2) gives 
\begin{equation}\label{171123D}
   \lambda_{1}(\T_{\sqrt{17}/4})\leq   \frac{\sqrt{17}}{2} -\frac{(5 \sqrt{17}-17)(1-3v_{3}^2)}{8(5+v_{3}^{2})}. 
\end{equation}
  As in the previous case, for $v^{2}_{3}=1/3$, inequality (\ref{171123D}) is nothing but  (Re), namely $\lambda_{1}(\T_{\sqrt{17}/4})<\frac{\sqrt{17}}{2}\approx 2.061552$, and for the choice $v_{3}^2 <1/3$, we have strict inequality in (\ref{171123D}). The small bound in  (\ref{171123C}) (resp. in (\ref {171123D})) is
  $
  \frac{17- \sqrt{17}}{8}
  $
  (resp. $\frac{\sqrt{17}}{2}- \frac{5 \sqrt{17}- 17}{40}$)
and it is achieved for $v_{3}^{2}=1$ (resp. for $v_{3}=0$). Therefore, we can assert 
$$
 \lambda_{1}(\T_{\sqrt{17}/4})\leq  \frac{17-\sqrt{17}}{8}\approx 1.609611.
$$

\smallskip

\textbf{ Case $R^2= 9/8$.} In this last situation, the quadratic form $Q_R$ identically vanishes and then,
all our bounds agree with  Reilly's one, that is
$$
\lambda_{1}(\T_{\frac{3}{2\sqrt{2}}})<\frac{\sqrt{2}}{2}.
$$

 \vspace{6mm}

\noindent {\bf Acknowledgment}
The first author was partially
supported by Spanish MICINN project PID2020-118452GB-I00. The second  named author was partially supported by the Spanish MICINN and ERDF project PID2020-116126GB-I00. Research partially supported by the ``Mar\'{\i}a de Maeztu'' Excellence Unit IMAG, reference CEX2020-001105-M, funded by
MCIN-AEI-10.13039-501100011033.

\end{document}